\documentclass[11pt]{article}
\usepackage{amssymb,amsmath,amsfonts,amsthm,graphics,ifthen,mathrsfs}
\usepackage[mathcal]{euscript}
\usepackage{epstopdf}
\usepackage{fullpage}
\usepackage[dvips]{graphicx}
\usepackage{setspace}

\usepackage[cmtip,arrow]{xy}

\newtheorem{theorem}{Theorem}
\newtheorem{lemma}[theorem]{Lemma}

\newtheorem{corollary}[theorem]{Corollary}

\newtheorem{claim}{Claim}[theorem]

\newtheorem{observation}[theorem]{Observation}

\newcommand{\cO}{{\mathcal O}}

\newcommand{\Aut}[0]{\operatorname{Aut}}

 \newenvironment{cem}
{
    \begin{enumerate}
        \setlength{\topsep}{0pt}
        \setlength{\parskip}{0pt}
        \setlength{\partopsep}{0pt}
        \setlength{\parsep}{0pt}         
        \setlength{\itemsep}{0pt} 
}
{
    \end{enumerate} 
}

\title{\bf List Distinguishing Parameters of Trees}
\author{Michael Ferrara$^{1,2}$, Ellen Gethner$^3$, Stephen G. Hartke$^{4,5}$,\\ Derrick Stolee$^{4,5,6}$ and Paul S. Wenger$^{1,7}$}
\date{\today}

\begin{document}

\maketitle

\begin{abstract}
A coloring of the vertices of a graph $G$ is said to be {\it distinguishing} provided no nontrivial automorphism of $G$ preserves all of the vertex colors.
The {\it distinguishing number} of $G$, $D(G)$, is the minimum number of colors in a distinguishing coloring of $G$.
The {\it distinguishing chromatic number} of $G$, $\chi_D(G)$, is the minimum number of colors in a distinguishing coloring of $G$ that is also a proper coloring. 

Recently the notion of a distinguishing coloring was extended to that of a list distinguishing coloring.
Given an assignment $L=\{L(v)\}_{v\in V(G)}$ of lists of available colors to the vertices of $G$, we say that $G$ is (properly) {\it $L$-distinguishable} if there is a (proper) distinguishing coloring $f$ of $G$ such that $f(v)\in L(v)$ for all $v$. The {\it list distinguishing number} of $G$, $D_{\ell}(G)$, is the minimum integer $k$ such that $G$ is $L$-distinguishable for any list assignment $L$ with $|L(v)|=k$ for all $v$.  Similarly, the {\it list distinguishing chromatic number} of $G$, denoted $\chi_{D_{\ell}}(G)$ is the minimum integer $k$ such that $G$ is properly $L$-distinguishable for any list assignment $L$ with $|L(v)|=k$ for all $v$.

In this paper, we study these distinguishing parameters for trees, and in particular extend an enumerative technique of Cheng  to show that for any tree $T$, $D_{\ell}(T)=D(T)$, $\chi_D(T)=\chi_{D_{\ell}}(T)$, and $\chi_D(T)\le D(T)+1$. \\ 

\end{abstract}


\footnotetext[1]{Department of Mathematical and Statistical Sciences, University of Colorado Denver, Denver, CO 80217.\\ \texttt{$\{$michael.ferrara;paul.wenger@ucdenver.edu$\}$}}
\footnotetext[2]{Research supported in part by Simons Foundation Grant \#206692.}
\footnotetext[3]{Department of Computer Science and Engineering, University of Colorado Denver, Denver, CO 80217.\\ \texttt{ellen.gethner@ucdenver.edu.}}
\footnotetext[4]{Department of Mathematics, University of Nebraska-Lincoln, Lincoln NE 68588.\\ \texttt{$\{$hartke;s-dstolee1$\}$@math.unl.edu}}
\footnotetext[5]{Research supported in part by NSF Grant DMS-0914815.}
\footnotetext[6]{Department of Computer Science and Engineering, University of Nebraska-Lincoln, Lincoln NE 68588.}
\footnotetext[7]{School of Mathematical Sciences, Rochester Institute of Technology, Rochester, NY 14623.}


\section{Introduction}

A {\it coloring} of a graph $G$ is a labeling $\phi:V(G)\rightarrow\mathbb{N}$; a {\it $k$-coloring} is a labeling $\phi:V(G)\rightarrow[k]$, where $[k]=\{1,2,\ldots,k\}$.
A coloring of the vertices of a graph $G$ is {\it distinguishing} if no nontrivial automorphism of $G$ preserves all of the vertex colors; such a coloring {\it distinguishes} $G$. In 1996 Albertson and Collins~\cite{AC} introduced the {\it distinguishing number $D(G)$} of a graph $G$, defined to be the minimum number of colors in a distinguishing coloring of $G$.  Since its introduction, the distinguishing number and related parameters have received considerable attention in the literature (see for example \cite{BogCow_Cube,Bout_DetCartProd,FI,KlWZhu_JAlg,KlZhu_CartPow}).  In 2006 Collins and Trenk~\cite{CT} introduced the {\it distinguishing chromatic number $\chi_D(G)$} of a graph $G$, defined to be the minimum number of colors in a distinguishing coloring of $G$ that is also a proper coloring.  Subsequent investigation of the distinguishing chromatic number (for instance, \cite{BhattLal_ChromDistPete,ChHartKaul_ChromPow,JerKlav_DistChromPow,LaFSey_DistBip}) has focused on similarities and disparities between the distinguishing chromatic number and one or both of the distinguishing number and the chromatic number.

One of the most studied variants of the chromatic number is the list chromatic number, introduced independently by Vizing in 1976~\cite{Vizing} and Erd\H{os}, Rubin, and Taylor in 1979~\cite{ERT}.  Recently Ferrara, Flesch, and Gethner~\cite{FFG} extended the notion of a distinguishing coloring to that of a list distinguishing coloring.  Given an assignment $L=\{L(v)\}_{v\in V(G)}$ of lists of available colors to the vertices of $G$, we say that $G$ is {\it $L$-distinguishable} if there is a distinguishing coloring $f$ of $G$ such that $f(v)\in L(v)$ for all $v$. The {\it list distinguishing number} of $G$, $D_{\ell}(G)$, is the minimum integer $k$ such that $G$ is $L$-distinguishable for any list assignment $L$ with $|L(v)|=k$ for all $v$.  The existence of a graph $G$ for which $D(G)$ and $D_{\ell}(G)$ are not equal remains a major open question.

While not explicitly introduced in \cite{FFG}, it is natural to consider a list analogue of the distinguishing chromatic number.  We say that $G$ is {\it properly $L$-distinguishable} if there is a distinguishing coloring $f$ of $G$ chosen from the lists such that $f$ is also a proper coloring of $G$.  The {\it list distinguishing chromatic number $\chi_{D_{\ell}}(G)$} of $G$ is the minimum integer $k$ such that $G$ is properly $L$-distinguishable for any assignment $L$ of lists with $|L(v)|\ge k$ for all $v$.  

In this paper, we study these four distinguishing parameters for trees, showing that $D_{\ell}(T)=D(T)$, $\chi_D(T)=\chi_{D_{\ell}}(T)$, and $\chi_D(T)\le D(T)+1$ for any tree $T$.  In our proofs, we extend an enumerative technique first introduced by Cheng to determine the distinguishing number of trees \cite{ChengTree}.
The technique has subsequently been used to determine the distinguishing number of both planar graphs \cite{ChengPlanar} and interval graphs \cite{ChengInterval}. 

\subsection{Preliminaries}

A vertex $v$ is in the {\it center} of a tree $T$ if $v$ minimizes $\max_{u\in V(T)} \{\operatorname{dist}(u,v)\}$.
In the majority of the paper, we will be considering rooted trees.  A {\it rooted tree} is a tree with an identified vertex called the {\it root}.
Given vertices $u$ and $v$ in a rooted tree $T$ with root $r$, $v$ is the {\it parent} of $u$ if $v$ is the neighbor of $u$ in the unique $u,r$-path in $T$, and $u$ in this case is a {\it child} of $v$.
Two vertices with the same parent are {\it siblings}.
More generally, $v$ is an {\it ancestor} of $u$ if $v$ lies anywhere in the unique $u,r$-path, and in this case $u$ is a {\it descendant} of $v$.
Given a vertex $v$ in a rooted tree $T$, we let $T_v$ denote the subtree of $T$ rooted at $v$ that is induced by $v$ and all of its descendants.
All automorphisms of a rooted tree stabilize the root.

Two colorings $\phi_1$ and $\phi_2$ of rooted trees $T_{v_1}$ and $T_{v_2}$ are {\it equivalent} if there exists an isomorphism $\pi:V(T_{v_1})\rightarrow V(T_{v_2})$ that maps $v_1$ to $v_2$ and that maps the coloring $\phi_1$ to $\phi_2$: that is, for every vertex $u\in V(T_{v_1})$, we have $\phi_1(u)=\phi_2(\pi(u))$.

\section{Distinguishing Colorings}

The main result of this section is the following.

\begin{theorem}\label{theorem:tree_dist}
If $T$ is a tree, then $D(T) = D_{\ell}(T)$.
\end{theorem}

To prove Theorem \ref{theorem:tree_dist}, we shall prove $D(T') = D_\ell(T')$ for any rooted tree $T'$;
this suffices since the following lemma provides a reduction from unrooted trees to rooted trees.

\begin{lemma}\label{lma:rootedunrooted}
	For any tree $T$, there is a rooted tree $T'$ with 
		$D(T) = D(T')$ and $D_\ell(T) = D_\ell(T')$.
\end{lemma}

\begin{proof}
In \cite{J}, Jordan proved that the center of a tree $T$ is either a single vertex or an edge, and in either case the center of $T$ is set-wise stabilized by all automorphisms. If the center of $T$ is a single vertex, then let $T'$ be $T$ rooted at the center and note that $\Aut(T) = \Aut(T')$.  

If the center of $T$ is the edge $e$, then let $T'$ be obtained by subdividing $e$ and rooting the resulting tree at the unique vertex in $T'-T$, which is the center of $T'$.  
In \cite{ChengTree}, Cheng showed that $D(T) = D(T')$, 
	and it is straightforward to show by a nearly identical argument that 
	$D_{\ell}(T) = D_{\ell}(T')$ as well. 
\end{proof}


The following observation describes a simple condition that is necessary and sufficient for a coloring to distinguish a rooted tree.

\begin{observation}\label{obs:distinguishsiblings}
	Let $T$ be a rooted tree. 
	For a fixed $x \in V(T)$, partition the children of $x$ into classes $C_1,\dots,C_t$ so that
		two vertices $u, v$ are in the same class if and only if $T_u \cong T_v$.
	A coloring $\phi$ distinguishes $T_x$ if and only if
	\begin{cem}
		\item for all children $v$ of $x$, the restriction of $\phi$ to $T_v$ distinguishes $T_v$, and 
		\item given vertices $u, v \in C_j$ for $j\in[t]$, the restrictions of $\phi$ to $T_u$ and $T_v$ are inequivalent.
	\end{cem}
\end{observation}

Following the notation from \cite{ChengTree}, we let $D(G; k)$ denote the number of equivalence classes of distinguishing $k$-colorings of a graph $G$.  Cheng showed that $D(T_x;k)$ can be computed recursively.

\begin{lemma}[Cheng \cite{ChengTree}]\label{lemma:chengrecursion}
	Let $T_x$ be a tree with root $x$ and partition 
		the children of $x$ into classes $C_1,\dots,C_t$ so that
		two vertices $u, v$ are in $C_j$ if and only if $T_u \cong T_v$. 
	Select a representative $u_j$ for each class $C_j$.
	The number of equivalence classes of distinguishing $k$-colorings of $T_x$ is given by
		\[ D(T_x; k) = k \prod_{j=1}^{t} { D(T_{u_j};k) \choose |C_j|}.\]
\end{lemma}


For a list assignment $L=\{L(v)\}_{v\in V(G)}$, we let $D(G; L)$ be the number of equivalence classes of distinguishing $L$-colorings of $G$.

\begin{theorem}\label{theorem:list_dist_count}
If $T_x$ is a rooted tree and $L=\{L(v)\}_{v\in V(T_x)}$ is a list assignment with $|L(v)| = k$ for all $v\in V(T_x)$, then $D(T_x;L) \ge D(T_x;k)$.
Equality holds if and only if $D(T_x;L)=0$ or $L(u) = L(v)$ whenever $u$ and $v$ lie in the same orbit of $T_x$.
\end{theorem}

\begin{proof}
Observe first that if $D(T_x;k) = 0$, then the conclusion holds trivially.
For any vertex $v\in T$, let $L_v$ denote the restriction of $L$ to the vertices in $T_v$.  
We proceed by induction to show that $D(T_v; L_v)\ge D(T_v; k)$ for each subtree $T_v$ of $T_x$.
If $v$ is a leaf, then $D(T_v,L_v)= D(T_v; k)=k$.

Assume that $v$ is not a leaf of $T_x$.
Partition the children of $v$ into classes $C_1,\dots,C_t$ as in Observation \ref{obs:distinguishsiblings}, and label the vertices in class $C_j$ as $v_{j,1}, v_{j,2},\dots, v_{j, m_j}$ where $m_j = |C_j|$.
By induction, $D(T_{v_{j,p}}; L_{v_{j,p}}) \geq D(T_{v_{j,p}}; k)$ for all $p\in[m_j]$.
Set $d_j = D(T_{v_{j,1}}; k)$; it follows that $d_j = D(T_{v_{j,p}}; k)$ for all $p \in [m_j]$ since the subtrees in $T_{v_{j,p}}$ are isomorphic.

For $p\in[m_j]$, let $S_{j,p}$ be a set of representative colorings from each equivalence class of distinguishing $L_{v_{j,p}}$-colorings
	of $T_{v_{j,p}}$; note that $|S_{j,p}|=D(T_{v_{j,p}}; L_{v_{j,p}})$, and by induction, $|S_{j,p}|\ge d_j$.
Let $R_j$ be the set of tuples $(\phi_{j,1},\phi_{j,2},\ldots,\phi_{j,m_j})\in S_{j,1}\times S_{j,2} \times \cdots \times S_{j,m_j}$
    such that $\phi_{j,p}$ and $\phi_{j,q}$ are inequivalent when $p\neq q$.
By Observation \ref{obs:distinguishsiblings}, an $L_v$-coloring $\phi$ distinguishes $T_v$ if and only if, for every class $C_j$,
	the tuple $(\phi_{j,1},\phi_{j,2},\ldots,\phi_{j,m_j})$ of colorings induced on $T_{v_{j,p}}$ is component-wise equivalent to an element of $R_j$.
Let $r_j$ be the number of equivalence classes in $R_j$, where two tuples are equivalent if they are the same up to the order of the coordinates.
A maximum set of inequivalent $L_v$-colorings of $T_v$ is formed by independently selecting one tuple $(\phi_{j,1},\phi_{j,2},\ldots,\phi_{j,m_j})$ from each equivalence class of $R_j$ for each class $C_j$ and selecting any color for $v$ from $L(v)$.
Hence, $D(T_v; L_v) = k \prod_{j=1}^t r_j$.
	
We bound $r_j$ by selecting colorings $(\phi_{j,1},\phi_{j,2},\ldots,\phi_{j,m_j})$ 
	for the subtrees $T_{v_{j,1}},\ldots,T_{v_{j,m_j}}$ in order.
By induction, for $p\in[m_j]$, there are at least $d_j - p + 1$ 
	choices for $\phi_{j,p}$ in $S_{j, p}$ that are inequivalent to the previous selections for $\phi_{j,1}, \dots, \phi_{j,p-1}$.
Thus, there are at least $d_j(d_j-1)\cdots (d_j-m_j+1)$ 
	ways to select $m_j$ inequivalent colorings $(\phi_{j,1},\phi_{j,2},\ldots,\phi_{j,m_j})$.
Each equivalence class in $R_j$ is counted at most $m_j!$ times,
	so 
	\[r_j \geq \frac{d_j(d_j-1)\cdots(d_j-m_j+1)}{m_j!} = {d_j \choose m_j} = {D(T_{v_{j,1}}; k)\choose |C_j|}.\]
Therefore by Lemma \ref{lemma:chengrecursion},
\[
	D(T_v; L_v) = k \prod_{j=1}^t r_j
		\geq k \prod_{j=1}^t {D(T_{v_{j,1}}; k) \choose |C_j|} 
		= D(T_v; k).
\]

If $D(T_x;L)=0$, then it is clear that $D(T_x;L)=D(T_x;k)$.
Otherwise we show that equality holds if and only if $L(u)=L(v)$ when $u$ and $v$ lie in the same orbit of $T_x$.
We prove this by induction on the number of vertices in $T_x$.
Let $L$ be a list assignment such that $D(T_x;L)>0$ and $D(T_x;L)=D(T_x;k)$.
The result holds trivially if $T_x$ has a single vertex.

Let $C_1,\ldots,C_t$ be the partition of the children of $x$ as above.
Equality holds for $T_x$ if and only if $r_j = {D(T_{v_{j,1}}; k) \choose |C_j|}$ for all $j\in[t]$.
Furthermore, $r_j = {D(T_{v_{j,1}}; k) \choose |C_j|} = \frac{d_j(d_j-1)\cdots(d_j-m_j+1)}{m_j!}$ if and only if $|S_{j,p}|=d_j$ for all $p\in[m_j]$ and every coloring in $S_{j,p}$ has an equivalent coloring in $S_{j,q}$ for all $p, q\in [m_j]$.
By induction, $|S_{j,p}|=d_j$ if and only if $L(y)=L(w)$ when $y$ and $w$ lie in the same orbit of $T_{v_{j,p}}$.

Consider a vertex $y$ in $T_{v_{j,p}}$.
By permuting colors on the orbit $\cO$ containing $y$, each color in $L(y)$ appears in $\cO$ in a distinguishing $L$-coloring $\phi_{j,p}$ in $S_{j,p}$.
The isomorphism $\sigma : V(T_{v_{j,p}}) \to V(T_{v_{j,q}})$ guaranteed by an equivalent pair $(T_{v_{j,p}}, \phi_{j,p}) \cong (T_{v_{j,q}}, \phi_{j,q})$ satisfies $\phi_{j,p}(y) = \phi_{j,q}(\sigma(y))$ for all $y \in V(T_{v_{j,p}})$.
Therefore, each color in $L(y)$ also appears in the lists of the vertices in the orbit of $\sigma(y)$ in $T_{v_{j,q}}$, and consequently $L(y) = L(\sigma(y))$.
Each vertex in the orbit of $y$ in $T_x$ that is not in $T_{v_{j,p}}$ lies in the image of $\cO$ under some such isomorphism $\sigma$.
Thus, every vertex in the orbit of $y$ in $T_x$ has the same list of colors as $y$.
Consequently, if equality holds, then $L(u)=L(v)$ whenever $u$ and $v$ lie in the same orbit of $T_x$.

Conversely, if $L(u)=L(v)$ whenever $u$ and $v$ lie in the same orbit of $T_x$, then $L(y)=L(w)$ when $y$ and $w$ lie in the same orbit of $T_{v_{j,p}}$.
Furthermore, each coloring $S_{j,p}$ has an equivalent coloring in $S_{j,q}$ for all $p, q\in [m_j]$.
\end{proof}

Note that $D(T)$ is the minimum $k$ so that $D(T; k)$ is positive.
Similarly, $D_\ell(T)$ is the minimum $k$ so that $D(T; L)$ is positive
	for every list assignment $L$ with $|L(v)| = k$.
By considering the list assignment that gives the same $k$ colors to every vertex, it is clear that $D_\ell(T)$ is positive only if $D(T)$ is.
Thus, the following corollary is immediate from Theorem \ref{theorem:list_dist_count}.

\begin{corollary}\label{cor:rootedlistequality}
If $T$ is a rooted tree, then $D(T) = D_\ell(T)$.
\end{corollary}		

Theorem \ref{theorem:tree_dist} follows from Lemma \ref{lma:rootedunrooted} and Corollary \ref{cor:rootedlistequality}.

\section{List Distinguishing Chromatic Number}

In this section, we prove that the distinguishing chromatic number of a tree is equal to the list distinguishing chromatic number using a similar enumerative method.

\begin{theorem}\label{thm:chromaticunrooted}
	If $T$ is a tree, then $\chi_{D}(T) = \chi_{D_\ell}(T)$.
\end{theorem}

For $x \in V(T)$, let $D_\chi(T_x; k)$ denote the number of equivalence classes of distinguishing proper $k$-colorings 
	of the rooted tree $T_x$.
For $i \in [k]$, let $D_\chi(T_x; k, i)$ denote the number of equivalence classes of distinguishing proper $k$-colorings 
	of $T_x$ in which $x$ gets color $i$.
Note that $D_\chi(T_x; k, i) = D_\chi(T_x; k, 1) = \frac{1}{k}D_\chi(T_x; k)$ for all $i \in [k]$.
Hence $D_\chi(T_x; k)=kD_\chi(T_x; k, 1)$.

Similarly, let $D_\chi(T_x; L)$ denote the number of equivalence classes of distinguishing proper $L$-colorings 
	of the rooted tree $T_x$.
For $i \in L(x)$,
	let $D_\chi(T_x; L, i)$ be the number of equivalence classes of 
	distinguishing proper $L$-colorings 
	of $T_x$ in which $x$ gets color $i$.
Here, the value $D_\chi(T_x; L, i)$ may change for different values of $i \in L(x)$.
Thus $D_\chi(T_x; L)=\sum_{i \in L(x)} D_\chi(T_x; L, i)$.

	
As we discuss below, consideration of rooted trees is sufficient to demonstrate $\chi_D(T) = \chi_{D_\ell}(T)$ in most, but not all unrooted cases.
The next result is analogous to Theorem \ref{theorem:list_dist_count}.


\begin{theorem}\label{theorem:list_chromatic_count}
	If $T_x$ is a rooted tree and $L = \{L(v)\}_{v\in V(T_x)}$ is
		a list assignment with $|L(v)| = k$ for all $v \in V(T_x)$, then
		 $D_\chi(T_x; L, i) \geq D_\chi(T_x; k, 1)$ for all $i \in L(x)$.
	Equality holds for all $i\in L(x)$ if and only if $k\ge 2$ and all lists in $L$ are identical.
\end{theorem}

\begin{proof}
	Similar to the proof of Theorem \ref{theorem:list_dist_count}, we show that
		$D_\chi(T_v;L_v,i) \geq D_\chi(T_v; k, 1)$ for all $v \in V(T_x)$ and $i \in L(v)$ inductively. 
	If $v$ is a leaf, then $D_\chi(T_v; L_v, i) = D_\chi(T_v; k, 1) = 1$.
	
Assume that $v$ is not a leaf.
	Partition the children of $v$ into equivalence classes $C_1, C_2, \dots, C_t$ 
		by isomorphism classes of subtrees.
	Label the vertices in each class $C_j$ as $v_{j,1}, v_{j,2},\dots, v_{j,m_j}$ 	
		where $m_j = |C_j|$.
	The following claim is analogous to Lemma \ref{lemma:chengrecursion}.
	
	\begin{claim}\label{claim:chicount}
		For $i \in [k]$, 
		$\displaystyle D_\chi(T_v; k, i) = \prod_{j=1}^t {(k-1)D_\chi(T_{v_{j,1}}; k, 1) \choose |C_j| }$.
	\end{claim}
	
	\begin{proof}
	By Observation~\ref{obs:distinguishsiblings}, a proper $k$-coloring $\phi$ such that $\phi(v)=i$ distinguishes $T_v$ 
		if and only if
		the colorings $\phi_{v_{j,p}}$, obtained by restricting $\phi$ to the trees $T_{v_{j,p}}$ for $v_{j,p} \in C_j$, are inequivalent distinguishing proper $k$-colorings
		with $\phi_{v_{j,p}}(v_{j,p}) \neq i$.
	Hence, there are $(k-1)D_\chi(T_{v_{j,1}}; k, 1)$ possible colorings 
		for $T_{v_{j,1}}$.
	Since $T_{v_{j,p}} \cong T_{v_{j,1}}$ for all $p\in[m_j]$, there are $(k-1)D_\chi(T_{v_{j,1}}; k, 1)$ possible colorings of $T_{v_{j,p}}$.
	We choose $|C_j|$ of these colorings to place on the trees $T_{v_{j,p}}$ for $v_{j,p} \in C_j$.
	Since $D_\chi(T_v; k, i)$ counts colorings up to isomorphism,
	the selection of these colorings is independent of their order in $C_j$.
	Also, the choices of these colorings are independent among the different isomorphism classes 
		$\{C_j\}_{j=1}^t$, so $D_\chi(T_v; k, i)$ is given by the product 
		$\prod_{j=1}^t { (k-1) D_\chi(T_{v_{j,1}}; k, 1) \choose |C_j| }$, 
		proving the claim.
	\end{proof}
	
	
Let $d_j^i = (k-1)D_\chi(T_{v_{j,1}}; k, 1)$.
	By induction, $$\sum_{i' \in L(v_{j,p})\setminus\{i\}} D_\chi(T_{v_{j,p}}; L_{v_{j,p}}, i') \geq (k-1)D_\chi(T_{v_{j,p}}; k, 1) = d_j^i.$$

We now consider the number of equivalence classes of distinguishing proper $L_{v_{j,p}}$-colorings of $T_{v_{j,p}}$ in which $v_{j,p}$ does not get color $i$.
Let $S_{j,p}^i$ be a set of representative colorings from each such equivalence class.
Let $R_j^i$ be the set of tuples $(\phi_{j,1},\phi_{j,2},\ldots,\phi_{j,m_j})\in S_{j,1}^i \times S_{j,2}^i \times \cdots \times S_{j,m_j}^i$
	    such that $\phi_{j,p}$ and $\phi_{j,q}$ are inequivalent when $p\neq q$.
By induction,
\[|S_{j,p}^i| = \sum_{i' \in L(v_{j,p}) \setminus\{i\}} D_\chi(T_{v_{j,p}}; L_{v_{j,p}}, i') \geq d_j^i.\]
	By Observation \ref{obs:distinguishsiblings}, 
		a proper $L_v$-coloring distinguishes $T_v$ if and only if, 
		for every class $C_j$, 
		the colorings $\phi_{j,p}$ induced on $T_{v_{j,p}}$
		form a tuple $(\phi_{j,1},\phi_{j,2},\ldots,\phi_{j,m_j})\in R_j^i$.
	Let $r_j^i$ be the number of equivalence classes in $R_j^i$, where
		two tuples are equivalent if they are the same up to the order of the coordinates.
	Form a maximum set of inequivalent proper $L_v$-colorings $\phi$ in which $v$ gets color $i$
		by independently selecting one tuple $(\phi_{j,1},\phi_{j,2},\ldots,\phi_{j,m_j})$
		from each equivalence class of $R_j^i$ for each class $C_j$.
	Hence, $D_\chi(T_v; L_v, i) = \prod_{j=1}^t r_j^i$.

	We bound $r_j^i$ by selecting colorings $(\phi_{j,1}, \dots, \phi_{j,m_j})$ for subtrees in order.
	For $p\in[m_j]$, there are at least $d_j^i-p+1$ selections for $\phi_{j,p}$ from $S_{j,p}$
	that are inequivalent to the selections of
			$\phi_{j,1}, \dots, \phi_{j,p-1}$.
	Thus, there are at least $d_j^i(d_j^i-1)\cdots(d_j^i - m_j +1)$ possible ways to select
		$m_j$ inequivalent colorings $\phi_{j,1},\phi_{j,2},\ldots,\phi_{j,m_j}$.
	Each equivalence class in $R_j^i$ is counted at most $m_j!$ times, so 
		\[r_j^i \geq \frac{d_j^i(d_j^i-1)\cdots(d_j^i-m_j+1)}{m_j!} = {d_j^i\choose m_j} = {(k-1)D_\chi(T_{v_{j,1}}; k, 1)\choose |C_j|}.\]
	Therefore, by Claim \ref{claim:chicount},
	\[ D_\chi(T_v; L, i) = \prod_{j=1}^t r_{j}^i \geq \prod_{j=1}^t {(k-1) D_\chi(T_{v_{j,1}};k,1)\choose |C_j|} = D_\chi(T_v; k, 1).\]

Consider the partition of the children of $x$ as in Observation~\ref{obs:distinguishsiblings}.
	Equality holds for all colors in $L(x)$ if and only if, given an arbitrary $i\in L(x)$, $r_j^i={(k-1)D_\chi(T_{v_{j,1}};k,1)\choose |C_j|}$ for all for all $j \in [t]$; this
	holds if and only if $|S_{j,p}^i| = d_j^i = (k-1)D_\chi(T_{v_{j,1}};k;1)$ for all $p\in[m_j]$.
	Noting that $(k-1)D_\chi(T_{v_{j,1}};k;1)=(k-1)D_\chi(T_{v_{j,p}};k;1)$, equality holds if and only if $D_\chi(T_{v_{j,p}};k;1)=D_\chi(T_{v_{j,p}};L;i')$ for all $i'\in L(v_{j,p})\setminus\{i\}$.
	Because there are multiple choices of $i\in L(x)$, it follows that equality holds for all colors in $L(x)$ if and only if $D_\chi(T_{v_{j,p}};k;1)=D_\chi(T_{v_{j,p}};L;i')$ for all $i'\in L(v_{j,p})$.
	Thus, by induction, the lists are identical in $T_{v_{j,p}}$.
Furthermore, $r_j^i =  {(k-1)D_\chi(T_{v_{j,1}};k,1)\choose |C_j|}$ for all $j \in [t]$ if and only if color $i$ appears in the list $L(v_{j,p})$ for all $p \in [m_j]$.
This is true for all $i\in L(x)$, so $x$ and all of its children have the same list.
%
%
%
%
%
\end{proof}

If $T_x$ is a rooted tree, then $\chi_D(T_x)$ is the minimum $k$ so that $D_\chi(T_x; k, 1)>0$, 
	and $\chi_{D_\ell}(T_x)$ is the minimum $k$ so that 
	for every assignment $L$ of lists of size $k$,
	there exists an $i \in L(x)$ such that
	$D_\chi(T_x; L, i)$ is positive.
By considering the list assignment in which every vertex gets the list $[k]$, the following corollary is immediate from Theorem \ref{theorem:list_chromatic_count}.

\begin{corollary}\label{cor:chirooted}
If $T_r$ is a rooted tree, then $\chi_D(T_r) = \chi_{D_\ell}(T_r)$.
\end{corollary}

The proof of Theorem \ref{theorem:tree_dist} followed from the reduction from unrooted to rooted trees in 
	Lemma \ref{lma:rootedunrooted} and the equality given in Corollary \ref{cor:rootedlistequality}.
For the distinguishing chromatic number and list distinguishing chromatic number,
such a reduction works in most cases.

\begin{lemma}\label{lemma:chiunrooted}
	If the center of $T$ is a single vertex or $\chi_D(T) \geq 3$, then $\chi_D(T) = \chi_{D_{\ell}}(T)$.
\end{lemma}

\begin{proof}
\noindent{\bf Case 1:} {\it The center of $T$ is a vertex $x$.}
	Since all automorphisms set-wise stabilize the center, 
		$x$ is stabilized by all automorphisms of $T$.
	Rooting the tree at $x$ does not change the automorphism group, so a $T_x$-distinguishing proper coloring of $V(T)$
		is also a $T$-distinguishing proper coloring.
	If $T'$ is $T$ rooted at $x$, then $\chi_D(T) = \chi_D(T') = \chi_{D_\ell}(T') = \chi_{D_\ell}(T)$.
		
\noindent{\bf Case 2:} {\it The center of $T$ is an edge $uv$.}
	Let $T'$ be obtained by subdividing $uv$ and rooting the resulting tree at 
		the unique vertex $x$ in $T' - T$, 
		which is the center of $T'$. 
	The automorphisms of $T'$ are given by the actions of automorphisms of $T$ on $V(T)$ while stabilizing $x$.
	
By Corollary~\ref{cor:chirooted}, $\chi_D(T') = \chi_{D_\ell}(T')$.
Since any $T$-distinguishing proper $k$-coloring $\phi$ uses at least $3$ colors, it is possible to extend $\phi$ to a $T'$-distinguishing proper $k$-coloring by assigning $x$ a color that is not $\phi(u)$ or $\phi(v)$.
Thus $\chi_D(T')\le \chi_D(T)$.

Since $T'$ contains an edge, $\chi_{D_\ell}(T')\ge 2$; let $k=\chi_{D_\ell}(T')$.
Consider a list assignment $L$ on $V(T')$ in which every list has size $k$, and let $L_T$ be the restriction of $L$ to $V(T)$.
Let $T'_v$ be the subtree of $T'$ rooted at $v$ consisting of all descendants of $v$; clearly $\chi_{D_\ell}(T'_v) \le \chi_{D_\ell}(T')$.
By Corollary~\ref{cor:chirooted}, $\chi_D(T'_v) = \chi_{D_\ell}(T'_v)$, so there is a $T'_v$-distinguishing proper $k$-coloring and we may assume that $v$ receives color $1$ in such a coloring.
Thus, by Theorem~\ref{theorem:list_chromatic_count}, there is a $T'_v$-distinguishing proper $L_v$-coloring in which $v$ gets color $i$ for each $i\in L(v)$.
Similarly, there is a $T'_u$-distinguishing proper $L_u$-coloring in which $u$ gets color $j$ for each $j\in L(u)$.
Because all lists have size at least $2$, it is possible to choose colors $i\in L(v)$ and $j\in L(u)$ so that $i\neq j$.
The combination of a distinguishing proper $L_v$-coloring of $T'_v$ in which $v$ receives color $i$ and a distinguishing proper $L_u$-coloring of $T'_u$ in which $u$ receives color $j$ produces a distinguishing proper $L_T$-coloring of $T$.
Since every list assignment on $V(T)$ is the restriction of some list assignment on $V(T')$, it follows that $\chi_{D_\ell}(T)\le \chi_{D_\ell}(T')$.

By considering the list assignment in which every vertex gets the same list of length $k$, we see that $\chi_{D}(T)\le \chi_{D_\ell}(T)$.
Therefore 
\[
	\chi_{D}(T)\le \chi_{D_\ell}(T)\le \chi_{D_\ell}(T')=\chi_D(T')\le \chi_D(T).
	\qedhere
\]
%
%
%
%
%
%
%
%
\end{proof}

\begin{lemma}\label{lemma:chiunrooted2}
	If $T$ is an unrooted tree with center $\{u,v\}$ 
		and $\chi_D(T) = 2$, then $\chi_{D_\ell}(T) = 2$.
\end{lemma}	

\begin{proof}
	Let $\phi$ be a proper 2-coloring of $T$.
	Note that $\phi(u) \neq \phi(v)$.
	Since $T$ is connected, the color at any vertex $y \in V(T)$ 
		is specified by the parity of the distance from $y$ to $u$:
		$\phi(y) = \phi(u)$ if and only if the distance from $y$ to $u$ is even.
	Hence, there are exactly two proper 2-colorings of $T$ and they are equivalent
		under a permutation of the colors.
	Since $\chi_D(T) = 2$, such a coloring must distinguish $T$.
	The center of $T$ is set-wise stabilized by all automorphisms of $T$, 
		so the distance from any vertex $y \in V(T)$ to $\{u, v\}$ is 
		preserved under all automorphisms of $T$.
	If $u$ and $v$ are point-wise stabilized by a non-trivial automorphism $\sigma$, then every $y \in V(T)$ is mapped to another element of the same distance to $u$.
Therefore $\phi(y) = \phi(\sigma(y))$ and this 2-coloring does not distinguish $T$.
	Thus, any non-trivial automorphism $\phi$ must swap $u$ and $v$.
The product of any two non-trivial automorphisms point-wise stabilize $u$ and $v$, and therefore the product is the identity.
	Hence, the automorphism group of $T$ is isomorphic to ${\mathbb Z}_2$ and any coloring
		$\phi$ where $\phi(u) \neq \phi(v)$ distinguishes $T$.
	Thus, $\chi_{D_\ell}(T) = 2 = \chi_D(T)$.
\end{proof}

Theorem~\ref{thm:chromaticunrooted} follows immediately from 
	Lemma~\ref{lemma:chiunrooted} and Lemma~\ref{lemma:chiunrooted2}.

\section{Distinguishing Chromatic Number}

In this section, we show that the distinguishing chromatic number and the distinguishing number of a tree differ by at most one and the bound is sharp.
We also describe when these parameters are different.

\begin{theorem} 
	If $T_x$ is a rooted tree, then $\chi_D(T_x) \leq D(T_x) + 1$.
\end{theorem}

\begin{proof}
	Let $\phi : V(T_x) \to \{1,\dots, k\}$ be a distinguishing $k$-coloring of $T_x$.
	We create a distinguishing proper $(k+1)$-coloring $\phi' : V(T) \to \{1,\dots, k, *\}$ 
		starting at the root: let $\phi'(x) = \phi(x)$.
	Proceed recursively; after coloring a vertex $v$, consider a child $u$ of $v$.
	If $\phi'(v) = \phi(u)$, then let $\phi'(u) = *$.
	Otherwise, let $\phi'(u) = \phi(u)$.
	Hence,  a vertex $u$ receives color $*$ if and only if its parent $v$
		has $\phi(v) = \phi(u)$ and the previous step assigned $\phi'(v) = \phi(v)$.
	Therefore, $\phi'$ is a proper coloring.
	Also note that if $x$ and $y$ are siblings with $\phi'(x) = \phi'(y)$, 
		then $\phi(x) = \phi(y)$.
	Since any nontrivial automorphism of $T_x$ interchanges subtrees rooted at siblings, any $\phi'$-preserving automorphism is also a $\phi$-preserving automorphism.
	Therefore $\phi'$ is distinguishing if $\phi$ is distinguishing.
\end{proof}

It follows from Lemmas \ref{lma:rootedunrooted}, \ref{lemma:chiunrooted}, and \ref{lemma:chiunrooted2} that $\chi_D(T) \leq D(T) + 1$ for unrooted trees as well.

We now characterize the trees for which $\chi_D(T) = D(T) + 1$.
Again we let $T'$ be the rooted tree obtained by rooting $T$ at its center if the center is unique, or rooting at the vertex obtained by subdividing the central edge of $T$.

\begin{theorem}\label{thm:certificate}
	If $T$ is a tree with $D(T)=k$, then $\chi_D(T) = k + 1$ if and only if 
		\begin{cem}
		\item $|V(T)|\ge 2$ and $k=1$, or 
		\item there exists a vertex $x \in V(T)$ and a set $S$ of children of $x$ in $T'$ so that the subtrees $T_u$ and $T_v$ are isomorphic for all $u, v \in S$ and $(k-1)D_\chi(T_u; k, 1) < |S|$ for $u \in S$.
		\end{cem}
\end{theorem}

\begin{proof}

If $D(T)=1$ and $T$ has more than one vertex, then clearly $\chi_D(T)=2$.
Thus we may assume that $D(T)\ge 2$, and we only consider graphs satisfying $\chi_D(T)\ge 3$.
By Lemmas~\ref{lma:rootedunrooted} and~\ref{lemma:chiunrooted}, it suffices to consider the rooted tree $T'$.

By Claim \ref{claim:chicount}, the number of proper distinguishing $k$-colorings of the rooted tree $T'_x$ is given by $kD_\chi(T'_x; k, 1) = k\prod_{j=1}^t {(k-1)D_\chi(T'_{v_j}; k, 1) \choose |C_j|}$.
There exists no proper distinguishing $k$-coloring of $T'$ if and only if this product is zero. 
This product is zero if and only if $(k-1)D_\chi(T'_{v_j}; k, 1) < |C_j|$ for some $j$.
If $(k-1)D_\chi(T'_{v_j}; k, 1) < |C_j|$ for some $j$, then letting $S=C_j$ suffices.
Conversely, if the set $S$ exists, then $S\subseteq C_j$ for some $j$.
If $v_j$ is an element of $S$, then $(k-1)D_\chi(T'_{v_j}; k, 1) < S\le|C_j|$, and there is no proper distinguishing $k$-coloring of $T'$.
%
%
%
%
%
%
\end{proof}

We conclude this section by noting that
Cheng's algorithm for computing $D(T)$ can be adapted to compute $\chi_D(T)$ 
	by replacing the counting method in Lemma \ref{lemma:chengrecursion}
	with Claim \ref{claim:chicount}.
This leads to a polynomial time\footnote{Cheng's algorithm runs in $O(n \log n)$ time and this modification replaces the formula in Lemma \ref{lemma:chengrecursion} with the formula in Claim \ref{claim:chicount}, giving $O(n\log n)$ time to compute $\chi_D(T)$.} algorithm to determine if $\chi_D(T) = D(T)$
	and if not, it can produce the certificate $(v, S)$ from Theorem \ref{thm:certificate}.


\bibliographystyle{plain}
\bibliography{distinguishing}

%
%
%
%
%
%

\end{document}